\documentclass[11pt,leqno]{amsart}
\usepackage{amsmath,amssymb,amsthm}
\newtheorem{theorem}{Theorem}[section]
\newtheorem{lemma}[theorem]{Lemma}

\newtheorem{proposition}[theorem]{Proposition}
\newtheorem{corollary}[theorem]{Corollary}
\theoremstyle{definition}

\theoremstyle{remark}

\numberwithin{equation}{section}

\def\fnote#1{\footnote}

\def\natu{{\mathbb N}}

\def\real{{\mathbb R}}

\def\ignora#1{}
\def\n3#1{\left\vert  \! \left\vert \! \left\vert \, #1 \, \right\vert \!
  \right\vert \! \right\vert }

\begin{document}

\title{Extreme differences between weakly open subsets and convex combinations of slices in Banach spaces}

\author{Julio Becerra Guerrero, Gin{\'e}s L{\'o}pez-P{\'e}rez and Abraham Rueda Zoca}
\address{Universidad de Granada, Facultad de Ciencias.
Departamento de An\'{a}lisis Matem\'{a}tico, 18071-Granada
(Spain)} \email{glopezp@ugr.es, juliobg@ugr.es,
arz0001@correo.ugr.es}

\thanks{The first author was partially supported by MEC (Spain) Grant MTM2011-23843 and Junta de Andaluc\'{\i}a grants
FQM-0199, FQM-1215. The second author was partially supported by
MEC (Spain) Grant MTM2012-31755 and Junta de Andaluc\'{\i}a Grant
FQM-185.} \subjclass{46B20, 46B22. Key words:
  slices, relatively weakly open sets, Radon-Nikodym property, renorming.}
\maketitle \markboth{J. Becerra, G. L\'{o}pez and A. Rueda   }{
  Extreme differences between weakly open subsets and convex combinations of slices }

\begin{abstract}
We show that  every Banach space containing isomorphic copies of
$c_0$ can be equivalently renormed
   so that every nonempty relatively
weakly open subset of its unit ball has diameter 2 and, however,
its unit ball still contains convex combinations of slices with
diameter arbitrarily small, which improves in a optimal way the
known results about the size of this kind of subsets in Banach
spaces.

\end{abstract}

\section{Introduction}
\par
\bigskip

The study of the size of slices, relatively weakly open subsets or
convex combinations of slices in the unit ball of a Banach space
is a relatively recent topic which has received intensive
attention in the last years. For example, in \cite{NyWe} it is
proved that the unit ball of every uniform algebra has all its
slices with diameter 2 and  in \cite{BLR} it is showed that the
unit ball of every non-hilbertizable real $JB^*$-triple has all
its relatively weakly open subsets with diameter 2. Many other
results in this direction have appeared \cite{BeLo,AcBeLo} giving
new geometrical properties in Banach spaces, extremely opposite to
the well known Radon-Nikodym property. See also \cite{ALN}. We
pass now to present these properties joint to its $w^*$-versions.

Given a Banach space $X$, $X$ is said to have the slice diameter 2
property (slice-D2P) if every slice in the unit ball of $X$ has
diameter 2. If every nonempty relatively weakly open subset,
respectively every convex combinations of slices, of the unit ball
of $X$ has diameter 2, we say that $X$ has the diameter 2 property
(D2P), respectively the strong diameter 2 property (strong-D2P).
 Also we define
the weak-star versions of the above properties, the
$w^*$-slice-S2P, $w^*$-D2P and $w^*$-strong-D2P property,
respectively, asking for the above conditions for $w^*$-slices,
nonempty relatively $w^*$-weakly open subsets and convex
combinations of $w^*$-slices of $B_{X^*}$, respectively.

It is clear that $(w^*)$-strong-D2P $\Rightarrow$ $(w^*)$-D2P
$\Rightarrow$ $(w^*)$-slice-D2P. In \cite{BeLoZo},  examples of
Banach spaces $X$ are exhibited satisfying the slice-D2P and
failing in an extreme way the D2P, in the sense that there are
nonempty relatively weakly open subsets in the unit ball with
arbitrarily small diameter. Then the biduals of these spaces,
$X^{**}$, are examples of  dual Banach spaces satisfying the
$w^*$-slice-D2P such that its unit ball contains nonempty
relatively weak-star open subsets with diameter arbitrarily small.

On the other hand there is a Banach space $X$ such that $X^*$
satisfies the $w^*$-strong-D2P, but its unit ball contains convex
combinations of slices with diameter arbitrarily small. Indeed,
take $X=C([0,1])$, the classical Banach space of continuous
functions on $[0,1]$ with the sup norm. Now, it is known that
$X^*=L_1[0,1]\oplus_1 Z$, for some subspace $Z$ of $X^*$ with RNP
\cite{BG}. Then the unit ball of $Z$ contains slices with
arbitrarily small diameter and so, $X^*$ also contains slices with
arbitrarily small diameter. On the other hand, $X$ has Daugavet
property, which implies that $X^*$ has $w^*$-strong-D2P
\cite[Lemma 2.3]{BeLoZo1}. Observe that now we have trivially that
$X^*$ has the $w^*$-slice-D2P and its unit ball contains slices
with diameter arbitrarily small and also $X^*$ has $w^*$-D2P and
its unit ball contains nonempty relatively weakly open subsets
with diameter arbitrarily small. Then the general situation is
shown in the following diagram

$$\begin{array}{ccccc}
{\rm Strong-D2P} & \stackrel{(1)}{\Rightarrow} & {\rm D2P} & \Rightarrow & {\rm slice-D2P}\\
\Downarrow &  & \Downarrow &  & \Downarrow\\
w^*-{\rm Strong-D2P} & \stackrel{(2)}{\Rightarrow} & w^*-{\rm
D2P}&\Rightarrow & w^*-{\rm slice-D2P}
\end{array}
$$

Following the above comments, we observe that all converse
implications, unless (1) and (2) are false in a extreme way, that
is, one can get diameter 2 for one of the properties in every
above pair and diameter arbitrarily small in the other one.

The aim of this note is to prove that ($w^*$)-D2P and
($w^*$)-strong-D2P are also extremely different in the above
sense, and so the converse implications (1) and (2) in the above
diagram are again false in a extreme way. Indeed,  we show in
Theorem~\ref{teorema} that there are Banach spaces $X$ with D2P
such that its unit ball contains convex combinations of slices
with diameter arbitrarily small. In fact every Banach space $X$
containing isomorphic copies of $c_0$ works. Then $X^{**}$ will be
an example of the extreme difference between $w^*$-D2P and
$w^*$-strong-D2P. Note that in \cite{AcBeLo}, it is proved that
$c_0\oplus_2 c_0$ is a Banach space with D2P and failing the
strong-D2P, but as we will see in Proposition~\ref{malo} every
convex combination of slices in the unit ball of $c_0\oplus_p c_0$
has diameter, at least, 1 for every $p\geq 1$.

We pass now to introduce some notation. For a Banach space $X$,
$X^*$ denotes the topological dual of $X$,  $B_X$ and $S_X$ stand
for the closed unit ball and unit sphere of $X$, respectively, and
$w$, respectively $w^*$, denotes the weak and weak-star topology
in $X$, respectively $X^*$. $[A]$ stands for the closed linear
span of the subset $A$ of $X$. We consider only real Banach
spaces. A slice of a set $C$ in $X$ is a set of $X$ given by
$$S=\{x\in C :x^*(x)>\sup x^*(C) -\alpha \}$$
where $x^*\in X^*$ and $0<\alpha  <\sup x^*(C)$. A $w^*$-slice of
a set $C$ of $X^*$ is a slice of $C$ determined by elements of
$X$, seen in $X^{**}$.

Recall that a slice of $B_X$ is a nonempty relatively weakly open
subset of $B_X$ and the family
$$\{\{x\in B_X: \vert x_i^*(x-x_0)\vert<\varepsilon,\ 1\leq i\leq
n\}:n\in\natu ,\ x_1^*,\cdots ,x_n^*\in X^*\}$$ is a basis of
relatively weakly open neighborhoods of $x_0\in B_X$. So every
relatively weakly open subset of $B_X$ has nonempty intersection
with $S_X$, whenever $X$ has infinite dimension.

Finally recall some connections between diameter 2 properties and
another well known geometrical properties in Banach spaces. Given
a Banach space $X$, $X$ is said to have the Daugavet property if
the equality $\Vert I+T\Vert=1+\Vert T\Vert$ holds for every
finite rank operator $T$ on $X$, where $I$ denotes the identity
operator on $X$. The norm of $X$ is said to be octahedral if for
every finite-dimensional subspace $F$ of $X$ and for every
$\varepsilon>0$ there is   $x\in S_X$ satisfying

$$\Vert y+\alpha x\Vert\geq (1-\varepsilon) (\Vert y\Vert +\vert \alpha\vert)\ \ \forall (y\in F, \alpha\in\mathbb R)$$
The norm of $X$ is called extremely rough if
$$~\mbox {lim sup }_{\Vert h\Vert \rightarrow 0}
{\Vert u+h\Vert +\Vert u-h\Vert -2\over \Vert h\Vert }~=2$$ for
every $u\in S_X$.

The Daugavet property implies the strong-D2P \cite{Sh}, the dual
of a Banach space with octahedral norm satisfies the
$w^*$-strong-D2P (see \cite{DGZ}) and the dual (or predual, if it
exists) of a Banach space with D2P has an extremely rough norm
\cite[Proposition I.1.11]{DGZ}.

\section{Main results}
\par
\bigskip

The following proposition shows that the space $c_0\oplus_p c_0$,
which has slice-D2P and fails the strong D2P \cite{AcBeLo}, is far
to satisfy that its unit ball contains convex combination of
slices with arbitrarily small diameter.

\begin{proposition}\label{malo} If $p\geq 1$, every convex combination of slices in
$B_{c_0\oplus_p c_0}$ has diameter at least 1.\end{proposition}

\begin{proof} Put $X=c_0\oplus_p c_0$ and consider
$\sum_{i=1}^n\lambda_i S(B_X , (x_i^*,y_i^*),\alpha_i)$ a convex
combination of slices in $B_X$, where $n\in\natu$, $0<\alpha_i<1$
for every $i$, $(x_i^*,y_i^*)\in S_{X^*}$ and $\lambda_i>0$ for
every $i$ with $\sum_{i=1}^n\lambda_i=1$. If
$\alpha=\min_i\alpha_i$, then $S_i\subset S(B_X ,
(x_i^*,y_i^*),\alpha_i)$, where $S_i=S(B_X,(x_i^*,y_i^*),\alpha)$
for every $i$. Now, given $\varepsilon>0$ arbitrary, for every
$1\leq i\leq n$ we choose $(x_i ,y_i)\in S_i$ such that $\Vert(x_i
,y_i)\Vert_X> 1-\varepsilon$ with  $A_i:=supp(x_i)$ finite and
$B_i:=supp(y_i)$ finite, where $supp(z)=\{n\in\natu :z(n)\neq 0$\}
for every $z\in c_0$. Pick $k_0\geq \max \cup_{i=1}^n
A_i\cup\cup_{i=1}^n B_i$ and $k>k_0$ such that $x_i\pm \Vert
x_i\Vert_{\infty}e_k$, $y_i\pm\Vert y_i\Vert_{\infty}e_k\in S_i$
for every $i$. From here we have that $$diam
(\sum_{i=1}^n\lambda_i S(B_X , (x_i^*,y_i^*),\alpha_i))\geq
diam(\sum_{i=1}^n\lambda_i S_i)\geq$$
$$2\Vert\sum_{i=1}^n\lambda_i(\Vert x_i\Vert_{\infty}e_k ,\Vert
y_i\Vert_{\infty}e_k)\Vert .$$ As $\Vert x_i\Vert_{\infty}^p+\Vert
y_i\Vert_{\infty}^p> 1-\varepsilon$ one has that for every $i$
either $\Vert x_i\Vert_{\infty}\geq
(\frac{1-\varepsilon}{2})^{1/p}$ or $\Vert y_i\Vert_{\infty}\geq
(\frac{1-\varepsilon}{2})^{1/p}$. Put $I=\{i:\Vert
x_i\Vert_{\infty}\geq (\frac{1-\varepsilon}{2})^{1/p}\}$ and
$t=\sum_{i\in I}\lambda_i$ ($t=0$ if $I=\emptyset$). Then $t\in
[0,1]$ and $1-t=\sum_{i\notin I}\lambda_i$. Now we have that
$$diam (\sum_{i=1}^n\lambda_i S(B_X , (x_i^*,y_i^*),\alpha_i))\geq
diam(\sum_{i=1}^n\lambda_i S_i)\geq$$
$$2\Vert\sum_{i=1}^n\lambda_i(\Vert x_i\Vert_{\infty}e_k ,\Vert
y_i\Vert_{\infty}e_k)\Vert \geq$$
$$ 2((\frac{t
(1-\varepsilon)^{1/p}}{2^{1/p}})^p+(\frac{(1-t)(1-\varepsilon)^{1/p}}{2^{1/p}})^p)^{1/p}=$$
$$\frac{2(1-\varepsilon)^{1/p}}{2^{1/p}}(t^p+(1-t)^p)^{1/p}\geq
\frac{2(1-\varepsilon)^{1/p}}{2^{1/p}}(\frac{1}{2^p}+\frac{1}{2^p})^{1/p}=(1-\varepsilon)^{1/p}.$$
Since $\varepsilon$ is arbitrary we get that $diam
(\sum_{i=1}^n\lambda_i S(B_X , (x_i^*,y_i^*),\alpha_i))\geq 1$ and
we are done. \end{proof}

Our first goal in order constructing a Banach space with D2P so
that its unit ball contains convex combinations of slices with
diameter arbitrarily small should be find out a closed, bounded
and absolutely convex subset with diameter 2 so that every
nonempty relatively weakly open subset has diameter 2 and
containing convex combinations of slices with diameter arbitrarily
small. We pass now to describe a family of closed, bounded and
convex subsets in $c_0$ with diameter 1 satisfying that every
nonempty relatively weakly open subset has diameter 1 and
containing convex combinations of slices with diameter arbitrarily
small.

Pick $\{\varepsilon_n\}$ an nonincreasing null scalars sequence.
We construct an increasing sequence of closed, bounded and convex
subsets $\{K_n\}$ in $c_0$ and a sequence $\{g_n\}$ in $c_0$ as
follows: Let $K_1=\{e_1\}$, $g_1=e_1$ and $K_2=co(e_1, e_1+e_2)$.
Choose $l_2>1$ and $g_2,\ldots ,g_{l_2}\in K_2$ a
$\varepsilon_2$-net in $K_2$. Assume that $n\geq 2$ and $m_n,\ l_n
,\ K_n$ and $\{g_1,\ldots ,g_{l_n}\}$ have been constructed, with
$K_n\subset B_{[e_1,\ldots ,e_{m_n}]}$ and $g_i\in K_n$ for every
$1\leq i\leq l_n$. Define $K_{n+1}$ as $$K_{n+1}=co(K_n\cup
\{g_i+e_{m_n+i}:1\leq i\leq l_n\}).$$ Let $l_{n+1}=m_n+l_n$ and
choose $\{g_{l_{n}+1},\ldots ,g_{l_{n+1}}\}\subset K_{n+1}$ so
that $\{g_1 ,\ldots ,g_{l_{n+1}}\}$ is a $\varepsilon_{n+1}$-net
in $K_{n+1}$. Finally we define $K_0=\overline{\cup_n K_n}$. Then
it follows that $K_0$ is a nonempty closed, bounded and convex
subset of $c_0$ such that $x(n)\geq 0$ for every $n\in \natu$ and
$\Vert x\Vert_{\infty}=1$ for every $x\in K_0$ and so
$diam(K_0)\leq 1$.

Now, if $i$ is fixed, we have from the construction that
$\{g_i+e_{m_n+i}\}_n$ is a sequence in $K_0$ weakly convergent to
$g_i$ and $\Vert (g_i-e_{m_n+i})-g_i\Vert=\Vert e_{m_n+i}\Vert=1$
for every $n$. Then $diam(K_0)=1$. We will use freely below the
subset $K_0$ and the above construction. Observe that, from the
above construction, one has that
$$K_0=\overline{\{g_i:i\in\natu\}}^{w}=\overline{\{g_i:i\in\natu\}}.$$

Mention that the construction of $K_0$ follows word for word the
definition of Poulsen simplex in $\ell_2$ \cite{P}, that is, the
unique, unless homeomorphism, Choquet simplex with a dense subset
of extreme points \cite{LO}. In fact, it is known \cite{ArOdRo}
that the weak-star closure of $K_0$ in $\ell_{\infty}$ is afinely
weak-star homeomorphic to the Poulsen simplex. However $K_0$ is
not a Choquet simplex, because it is not weakly compact, $K_0$ is
a simplex in a more general definition than Choquet simplex.

Let us see that $K_0$ satisfies the requirements we are looking
for.

\begin{proposition}\label{k0} $K_0$ is a closed, bounded and convex subset of
$c_0$ with $diam(K_0)=1$ satisfying that every nonempty relatively
weakly open subset of $K_0$ has diameter 1 and $K_0$ contains
convex combinations of slices with diameter arbitrarily
small.\end{proposition}

\begin{proof} The fact that $K_0$ is a closed, bounded and convex subset of
$c_0$ with $diam(K_0)=1$ have been proved after the construction
of $K_0$. From \cite[Theorem 1.2]{ArOdRo}, we deduce that $K_0$
has convex combinations of slices with diameter arbitrarily small.
Now pick $U$ a nonempty relatively weakly open subset of $K_0$.
From the construction of $K_0$ we noted that
$K_0=\overline{\{g_i:i\in\natu\}}^{w}$ and so there is $i\in\natu$
such that $g_i\in U$. Now, again from the construction of $K_0$,
$g_i+e_{m_n+i}\in K_0$ for every $n$. Thus, $g_i+e_{m_n+i}\in U$
for every $n$ grater than some $n_0$, since $\{g_i+e_{m_n+i}\}_n$
is weakly convergent to $g_i$. Therefore, $diam(U)\geq \Vert
e_{m_n+i}\Vert =1$. \end{proof}

Our next goal should be to get from $K_0$ a closed, absolutely
convex, bounded subset with diameter 2, containing convex
combinations of slices with diameter arbitrarily small and so that
every nonempty relatively weakly open subset has diameter 2. For
this, we see $K_0$ as a subset of $c$, the space of scalars
convergent sequence with the sup norm and define
$$K=2\overline{co}((K_0-\frac{{\bf 1}}{2})\cup (-K_0+\frac{{\bf
1}}{2})),$$ where ${\bf 1}$ is the sequence of $c$ with every
coordinate equal 1. Now, it is clear that $K$ is a closed,
absolutely convex and bounded subset of $c$ with $diam(K)=2$.

Our next point is constructing a Banach space with D2P and so that
its unit ball contains convex combinations of slices with diameter
arbitrarily small. It is natural to think that this Banach space
is some renorming  of $c$, which would be in fact a renorming of
$c_0$. For this we need the following lemmas.

\begin{lemma}\label{lemac0} Let $X$ be a Banach space containing an isomorphic
copy of $c_0$. Then there is an equivalent norm $| \|
 \cdot\| | $ in $X$ satisfying that $(X
,|\|\cdot\| |)$ contains an isometric copy of $c$ and for every
$x\in B_{(X,|\|\cdot\| |)}$ there are sequences $\{x_n\}$,
$\{y_n\}\in B_{(X,|\|\cdot\| |)}$ weakly convergent to $x$ such
that $ | \| x_n -y_n\| | =2$ for every $n\in \natu$. In fact,
$x_n=x+(1-\alpha_n)e_n$ and $y_n=x-(1+\alpha_n)e_n$ for some
scalar sequence $\{\alpha_n\}$ with $\vert \alpha_n\vert\leq 1$
for every $n$.  \end{lemma}
\begin{proof}As $X$ contains isomorphic copies of $c$, we can assume that $c$ is, in fact,
 an isometric subspace of $X$. Then for
every $Y$ separable subspace of $X$ containing $c$, there is a
linear and continuous projection $P_Y:Y\longrightarrow c$ with
$\Vert P_Y\Vert\leq 8$. Indeed, let us consider the onto linear
isomorphism $T:c\longrightarrow c_0$ given by
$T(x)(1)=\frac{1}{2}\lim_n x(n)$ and
$T(x)(n)=\frac{1}{2}(x(n)-\lim_n x(n))$ for every $n>1$. Note that
$\Vert T\Vert =1$ and $\Vert T^{-1}\Vert=4$. On the other hand, by Sobczyk Theorem, there exists a linear projection $\pi:Y\rightarrow c_0$ such that $\Vert \pi\Vert \leq 2$. Now $P_Y=T^{-1}\circ \pi$ satisfies $\Vert P_Y\Vert\leq 8$ and is the required projection from $Y$ onto $c$.  

Let $\Upsilon$ be the family of subspaces $Y$ of $X$ containing
$c$ such that $c$ has finite codimension in $Y$. Consider the
filter basis $\Upsilon$ given by $\{Y\in \Upsilon: Y_0\subset
Y\}$, where $Y_0\in \Upsilon$ and call ${\mathcal U}$ the
ultrafilter containing the generated filter by the above filter
basis.

For every $Y\in \Upsilon$, we define a new norm in $X$ given by
$$\Vert x\Vert_Y:=\max\{\Vert P_Y(x)\Vert ,\Vert x-P_Y(x)\Vert\}.$$
Finally, we define the norm on $X$ given by $|\| x\| |
:=\lim_{\mathcal U}\Vert x\Vert_Y$. Observe that $\frac{1}{8}\Vert
x\Vert\leq |\| x\| |\leq 3\Vert x\Vert$ for every $x\in X$ and so
$|\|\cdot\| |$ is an equivalent norm in $X$ such that $|\| x\| |
=\Vert x\Vert_{\infty}$ for every $x\in c$, where
$\Vert\cdot\Vert_{\infty}$ is the sup norm in $c$. Hence $(X,|\|
\cdot\| |)$ contains an isometric copy of $c$.

Pick $x_0 \in B_{(X,|\| \cdot\| |)}$. In order to prove the
remaining statement let $\{e_n\}$ and $\{e_n^*\}$ the usual basis
of $c_0$ and the biorthogonal functionals sequence, respectively.

Choose $\lambda\in \real$ and $n\in \natu$. For every
$Y\in\Upsilon$ with $x_0\in Y$ we have that $$\Vert x_0+\lambda
e_n\Vert_Y =\max\{\Vert P_Y(x_0)+\lambda e_n\Vert ,\Vert
x_0-P_Y(x_0)\Vert\}=$$
$$\max\{\vert\lambda+e_n^*(P_Y(x_0))\vert ,\Vert
P_Y(x_0)-e_n^*(P_Y(x_0))e_n\Vert ,\Vert x_0-P_Y(x_0)\Vert\}.$$

Call  $\beta_n =\lim_{\mathcal U}\max\{ \Vert
P_Y(x_0)-e_n^*(P_Y(x_0))e_n\Vert ,\Vert x_0-P_Y(x_0)\Vert\}$
and\break $\alpha_n=\lim_{\mathcal U}e_n^*(P_Y(x_0))$. Then $\vert
\Vert x_0+\lambda e_n\vert \Vert=\max\{\vert \lambda
+\alpha_n\vert , \beta_n\}$. Note that $\vert \alpha_n\vert\leq 1$
and $\beta_n\leq 1$ since $|\| x_0\| |\leq 1$.

Doing $x_n:=x_0+(1-\alpha_n)e_n$ and $y_n:=x_0-(1+\alpha_n)e_n$
for every $n$, we get that $x_n$, $y_n\in B_{(X,|\| \cdot\| | )}$.
Finally, it is clear that $\{x_n\}$ and $\{y_n\}$ are weakly
convergent sequences to $x_0$ and $|\| x_n-y_n\| |=2$ for every
$n\in \natu$.\end{proof}

\begin{lemma}\label{lema} Let $X$ be a vector space and $A$, $B$
convex subsets of $X$ such that $\frac{A-A}{2}\subset B$. Then
$$co(A\cup -A\cup B)=co(A\cup B)\cup co(-A\cup B).$$
\end{lemma}

\begin{proof} It is enough to prove that
$$co(A\cup -A\cup B)\subset co(A\cup B)\cup co(-A\cup B).$$

For this, take  $x\in co(A\cup -A\cup B)$. As $A$ and $B$ are
convex subsets we get that $x=\lambda_1 a_1+\lambda_2
(-a_2)+\lambda_3 b$, where $a_1,a_2\in A$, $b\in B$ and $\lambda_1
,\lambda_2, \lambda_3\in [0,1]$ with
$\lambda_1+\lambda_2+\lambda_3=1$.

Assuming that $\lambda_1\geq \lambda_2$, one has that
$$x=(\lambda_1-\lambda_2) a_1+2\lambda_2\frac{a_1-a_2}{2}
+\lambda_3 b .$$ Then $x$ is a convex combination of elements in
$A\cup B$, since from hypotheses $\frac{a_1-a_2}{2}\in B$, and so
$x\in co(A\cup B)$.

If $\lambda_1\leq\lambda_2$, one has similarly that $x\in
co(-A\cup B)$.

In any case, $x\in co(A\cup B)\cup co(-A\cup B)$ and we are
done.\end{proof}

It would be natural to think that some renorming of $c_0$ gives us
our goal space.   The following result shows that this is true for
every Banach space containing $c_0$.

\begin{theorem}\label{teorema} Let $X$ be a Banach space containing isomorphic
copies of $c_0$. Then there is an equivalent norm $\Vert \vert
\cdot \vert \Vert$ in $X$ such that every nonempty relatively
weakly open subset of $B_{(X,\Vert \vert \cdot \vert \Vert)}$ has
diameter 2 and $B_{(X,\Vert \vert \cdot \vert \Vert)}$ contains
convex combinations of slices with diameter arbitrarily
small.\end{theorem}

\begin{proof} From the  Lemma~\ref{lemac0}, we can assume that $X$
contains an isometric copy of $c$ and  for every $x\in B_{X}$
there are sequences $\{x_n\}$, $\{y_n\}\in B_{X}$ weakly
convergent to $x$ such that $ \Vert x_n -y_n\Vert =2$ for every
$n\in \natu$.

Fix $0<\varepsilon <1$ and consider in $X$ the equivalent norm
$\Vert \cdot \Vert_{\varepsilon}$ whose unit ball is
$B_{\varepsilon}=\overline{co}(2(K_0-\frac{{\bf 1}}{2})\cup
2(-K_0+\frac{{\bf 1}}{2})\cup [(1-\varepsilon )B_X+\varepsilon
B_{c_0}])$. Then we have $\Vert x\Vert \leq\Vert
x\Vert_{\varepsilon}\leq \frac{1}{1-\varepsilon}\Vert x\Vert$ for
every $x\in X$ and $\Vert x\Vert=\Vert x\Vert_{\infty}$ for every
$x\in c$.

Fix $\gamma >0$. From Proposition~\ref{k0}, there exist
$S_1,\cdots ,S_n$ slices of $K_0$ such that
$$\dim (\frac{1}{n}\sum_{i=1}^nS_i )<\frac{1}{4}(1-\varepsilon )\gamma .$$ We
can assume that $S_i=\{x\in K: x_i^*(x)>1-\widetilde{\delta }\}$
where $x_i^*\in c^*$ and $\sup x_i^*(K_0)=1$  for every
$i=1,\ldots ,n$ and $0< \widetilde{\delta } <1$. Denote by ${\bf
1}$ the sequence in $c$ with all its coordinates equal 1. It is
clear that $\sup x_i^*(2(K_0-\frac{{\bf
1}}{2}))=2(1-x_i^*(\frac{{\bf 1}}{2}))$, for all $i=1,\cdots ,n$.
We put $\rho ,\delta>0$ such that $\frac{1}{2}\rho \Vert
x_i^*\Vert +\delta <\widetilde{\delta }$, $2\rho <\varepsilon$,
$\rho \Vert x_i^*\Vert < 4\delta$, and $\frac{(7-2\varepsilon
)\rho }{(1-\varepsilon )}< \gamma $, for all $i=1,\ldots ,n$. We
consider the relatively weakly open set of $B_\varepsilon $ given
by
$$ U_i:=\{x\in B_\varepsilon :x_i^*(x)>2(1-\delta -x_i^*(\frac{{\bf 1}}{2}))+\frac{1}{2}\rho \Vert x_i^*\Vert , \ \lim_k
x(k)<-1+\rho^2\}$$ for every $i=1,\ldots ,n$, where $x_i^*$ and
$\lim_n$ denote the Hahn-Banach extensions to $X$ of the
corresponding functionals on $c$. It is clear that $\Vert
x_i^*\Vert_{\varepsilon}=\Vert x_i^*\Vert $ for every $i=1,\ldots
, n$ and $\Vert \lim_n\Vert_{\varepsilon}=\Vert \lim_n \Vert =1$.

Since $\rho \Vert x_i^*\Vert < 4\delta$, we have that
$2(1-x_i^*(\frac{{\bf 1}}{2}))>2(1-\delta -x_i^*(\frac{{\bf
1}}{2}))+\frac{1}{2}\rho \Vert x_i^*\Vert$. Now, we have that
$\sup x_i^*(2(K_0-\frac{{\bf 1}}{2}))=2(1-x_i^*(\frac{{\bf
1}}{2}))$, then there exist $x\in K_0$ such that
$x_i^*(2(x-\frac{{\bf 1}}{2}))>2(1-\delta -x_i^*(\frac{{\bf
1}}{2}))+\frac{1}{2}\rho \Vert x_i^*\Vert$ and $\lim_k 2(
x(k)-\frac{1}{2})=-1<-1+\rho^2$. This implies that $U_i  \neq
\emptyset$ for every $i=1,\ldots ,n$. In order to estimate the
diameter of $\frac{1}{n}\sum_{i=1}^nU_i $, it is enough to compute
the diameter of
$$\frac{1}{n}\sum_{i=1}^nU_i \cap co(2(K_0- \frac{{\bf 1}}{2})\cup
-2(K_0- \frac{{\bf 1}}{2})\cup [(1-\varepsilon) B_X+\varepsilon
B_{c_0}]).$$ Since $2(K_0- \frac{{\bf 1}}{2})$ and $(1-\varepsilon
)B_X+\varepsilon B_{c_0}$ are a convex subsets of $B_\varepsilon$,
given $x\in B_\varepsilon$, we can assume that $x=\lambda _1 2(a-
\frac{{\bf 1}}{2})+\lambda _2 2 (-b+ \frac{{\bf 1}}{2})+\lambda _3
[(1-\varepsilon )x_0+\varepsilon y_0]$, where $\lambda _i\in
[0,1]$ with $\sum_{i=1}^3\lambda_i=1$ and $a,b\in K_0$, $x_0\in
B_X$, and $y_0\in B_{c_0}$.

Given $x,y\in \frac{1}{n}\sum_{i=1}^nU_i$, for $i=1,\cdots , n$,
there exist $a_i,a'_i,b_i,b'_i\in K_0$, $\lambda _{(i,j)},\lambda
'_{(i,j)}\in [0,1]$ with $j=1,2,3$ and, $x_i, x_i'\in B_X$, and
$y_i,y_i'\in B_{c_0}$, such that,
$$2\lambda_{(i,1)} (a_i-\frac{{\bf
1}}{2})+2\lambda_{(i,2)}(-b_i+\frac{{\bf
1}}{2})+\lambda_{(i,3)}[(1-\varepsilon )x_i+\varepsilon y_i]$$
$$2\lambda_{(i,1)}' (a_i-\frac{{\bf
1}}{2})+2\lambda_{(i,2)}'(-b_i+\frac{{\bf
1}}{2})+\lambda_{(i,3)}'[(1-\varepsilon )x_i'+\varepsilon y_i']$$
belong to $U_i$ and
$$x=\frac{1}{n}\sum_{i=1}^n 2\lambda_{(i,1)} (a_i-\frac{{\bf
1}}{2})+2\lambda_{(i,2)}(-b_i+\frac{{\bf
1}}{2})+\lambda_{(i,3)}[(1-\varepsilon )x_i+\varepsilon y_i]$$ and
$$y=\frac{1}{n}\sum_{i=1}^n 2\lambda_{(i,1)}' (a_i-\frac{{\bf
1}}{2})+2\lambda_{(i,2)}'(-b_i+\frac{{\bf
1}}{2})+\lambda_{(i,3)}'[(1-\varepsilon )x_i'+\varepsilon y_i'].$$

For $i=1,\ldots , n$, we have that $$2\lambda_{(i,1)}
(a_i-\frac{{\bf 1}}{2})+2\lambda_{(i,2)}(-b_i+\frac{{\bf
1}}{2})+\lambda_{(i,3)}[(1-\varepsilon )x_i+\varepsilon y_i]\in
U_i,$$ then $$\lim_k(2\lambda_{(i,1)} (a_i-\frac{{\bf
1}}{2})+2\lambda_{(i,2)}(-b_i+\frac{{\bf
1}}{2})+\lambda_{(i,3)}[(1-\varepsilon )x_i+\varepsilon
y_i])<-1+\rho^2.$$

This implies that
$$2\lambda _{(i,2)}+\lambda _{(i,3)}\varepsilon -1= -\lambda _{(i,1)}+\lambda _{(i,2)}
-\lambda _{(i,3)}(1-\varepsilon )<-1+\rho ^2.$$ Since
$2\rho<\varepsilon$, we deduce that
$\lambda_{(i,2)}+\lambda_{(i,3)}<\frac{1}{2}\rho$. As a
consequence we get that

\begin{equation}\label{landa}
\lambda_{(i,1)}>1-\frac{1}{2}\rho , \end{equation}
 and similarly we get that
\begin{equation}\label{landaprima}
\lambda'_{(i,1)}>1-\frac{1}{2}\rho ,
\end{equation}
for every $i=1,\ldots ,n$. Now, applying \ref{landa}, and
\ref{landaprima}, we have that
$$\Vert x-y\Vert_\varepsilon \leq \frac{1}{n}\Vert \sum_{i=1}^n  2\lambda_{(i,1)} (a_i-\frac{{\bf
1}}{2})-2\lambda_{(i,1)}' (a_i'-\frac{{\bf 1}}{2})\Vert
_\varepsilon + $$
$$\frac{1}{n}\sum_{i=1}^n \Vert
2\lambda_{(i,2)}(-b_i+\frac{{\bf 1}}{2})\Vert _\varepsilon +
\frac{1}{n}\sum_{i=1}^n \Vert 2\lambda_{(i,2)}'(-b_i'+\frac{{\bf
1}}{2})\Vert _\varepsilon +$$
$$\frac{1}{n}\sum_{i=1}^n \Vert
\lambda_{(i,3)}[(1-\varepsilon )x_i+\varepsilon
y_i]\Vert_\varepsilon + \frac{1}{n}\sum_{i=1}^n \Vert
\lambda_{(i,3)}'[(1-\varepsilon )x_i'+\varepsilon y_i']\Vert
_\varepsilon \leq $$

$$\frac{1}{n}\Vert \sum_{i=1}^n
2\lambda_{(i,1)} (a_i-\frac{{\bf 1}}{2})-2\lambda_{(i,1)}'
(a_i'-\frac{{\bf 1}}{2})\Vert _\varepsilon +$$
$$\frac{1}{n}\sum_{i=1}^n
(\lambda_{(i,2)}+\lambda_{(i,3)})+\frac{1}{n}\sum_{i=1}^n
(\lambda_{(i,2)}'+\lambda_{(i,3)}')\leq$$

$$\frac{1}{n}\Vert \sum_{i=1}^n
 2\lambda_{(i,1)} (a_i-\frac{{\bf 1}}{2})-2\lambda_{(i,1)}'
(a_i'-\frac{{\bf 1}}{2})\Vert _\varepsilon +\rho \leq $$

$$ \frac{2}{n}\Vert \sum_{i=1}^n
 \lambda_{(i,1)} a_i-\lambda_{(i,1)}' a_i'\Vert _\varepsilon +
\frac{1}{n}\sum_{i=1}^n \vert
\lambda_{(i,1)}-\lambda_{(i,1)}'\vert \Vert {\bf
1}\Vert_\varepsilon +\rho \leq $$

$$ \frac{2}{n}\Vert \sum_{i=1}^n
\lambda_{(i,1)} a_i-\lambda_{(i,1)}' a_i'\Vert _\varepsilon +
\frac{(3-2\varepsilon)}{2(1-\varepsilon)}\rho .
$$

Now  $$\Vert \sum_{i=1}^n \lambda_{(i,1)}
a_i-\lambda_{(i,1)}'a'_i\Vert_\varepsilon \leq $$ $$ \Vert
\sum_{i=1}^n (\lambda_{(i,1)}-1) a_i\Vert_\varepsilon +\Vert
\sum_{i=1}^n a_i-a'_i\Vert_\varepsilon +\Vert \sum_{i=1}^n
(\lambda'_{(i,1)}-1) a'_i\Vert_\varepsilon \leq$$

$$\frac{1}{1-\varepsilon}\Vert \sum_{i=1}^n a_i-a'_i\Vert  +\sum_{i=1}^n
\frac{1}{1-\varepsilon}\vert\lambda_{(i,1)}-1\vert \Vert a_i\Vert
+\sum_{i=1}^n \frac{1}{1-\varepsilon}\vert\lambda_{(i,1)}'-1\vert
\Vert a_i'\Vert \leq$$
$$\frac{1}{1-\varepsilon}\Vert \sum_{i=1}^n a_i-a'_i\Vert  +\frac{1}{1-\varepsilon}n\rho .$$

We deduce that
\begin{equation}\label{acotado}
\Vert x-y\Vert_\varepsilon \leq \frac{2}{1-\varepsilon}\Vert
\frac{1}{n} \sum_{i=1}^n a_i-a'_i\Vert
+\frac{(7-2\varepsilon)}{2(1-\varepsilon )}\rho .
\end{equation}
On the other hand, we have that, for every $i=1,\ldots,n$,

$$x_i^*(2\lambda_{(i,1)}
(a_i-\frac{{\bf 1}}{2})+2\lambda_{(i,2)}(-b_i+\frac{{\bf
1}}{2})+\lambda_{(i,3)}[(1-\varepsilon )x_i+\varepsilon
y_i])>$$$$2(1-\delta -x_i^*(\frac{{\bf 1}}{2}))+\frac{}{}\rho
\Vert x_i^*\Vert ,$$ then

$$x_i^*(2\lambda_{(i,1)}
(a_i-\frac{{\bf 1}}{2}))+\frac{1}{2}\rho \Vert x_i^*\Vert \geq$$

$$x_i^*(2\lambda_{(i,1)}
(a_i-\frac{{\bf 1}}{2}))+\lambda_{(i,2)}\Vert x_i^*\Vert
_\varepsilon +\lambda_{(i,3)}\Vert x_i^*\Vert_\varepsilon \geq$$

$$x_i^*(2\lambda_{(i,1)} (a_i-\frac{{\bf
1}}{2})+2\lambda_{(i,2)}(-b_i+\frac{{\bf
1}}{2})+\lambda_{(i,3)}[(1-\varepsilon )x_i+\varepsilon y_i]).$$

We have that $$x_i^*(2\lambda_{(i,1)} (a_i-\frac{{\bf
1}}{2}))>2(1-\delta -x_i^*(\frac{{\bf 1}}{2})),$$ and hence
$$x_i^*(\lambda_{(i,1)}a_i)>1-\delta -(1-\lambda_{(i,1)})x_i^*(\frac{{\bf 1}}{2}))\geq  1-\delta -\frac{1}{2}\rho \Vert x_i^*\Vert
.$$We recall that $\delta +\frac{1}{2}\rho \Vert x_i^*\Vert <
\widetilde{\delta } $, then
$x_i^*(\lambda_{(i,1)}a_i)>1-\widetilde{\delta }$. It follows that
$x_i^*(a_i)>1-\widetilde{\delta }$. Now $a_i\in K_0\cap S_i$, and
similarly we get that $a'_i\in K_0\cap S_i$, for every $i=1,\ldots
,n$, and $\frac{1}{n}\sum_{i=1}^n a_i,\ \frac{1}{n}\sum_{i=1}^n
a'_i\in \frac{1}{n}\sum_{i=1}^n S_i$. Since the diameter of
$\frac{1}{n}\sum_{i=1}^n S_i$ is less than
$\frac{1}{4}(1-\varepsilon )\gamma $, we deduce that
$\frac{1}{n}\Vert \sum_{i=1}^n a_i-a'_i\Vert
<\frac{1}{4}(1-\varepsilon )\gamma $. Finally, we conclude from
\ref{acotado} and the above estimation that $$ \Vert
x-y\Vert_\varepsilon \leq \gamma .$$ Hence the set
$\frac{1}{n}\sum_{i=1}^n U_i$ has diameter, at most $\gamma$, for
the norm $\Vert\cdot\Vert_{\varepsilon}$. We recall now that every
relatively weakly open subset of $B_\varepsilon$ contains a convex
combination of slices \cite[Lemme 5.3]{BB}. So we conclude that
$B_\varepsilon$ has convex combinations of slices with diameter
arbitrarily small.

In order to prove that every nonempty relatively weakly open
subset of $B_\varepsilon$ has diameter 2, we recall that
$K_0=\overline{\{g_i:i\in\natu\}}$.

Recall  that $B_{\varepsilon}=\overline{co}(2(K_0-\frac{{\bf
1}}{2})\cup 2(-K_0+\frac{{\bf 1}}{2})\cup [(1-\varepsilon
)B_X+\varepsilon B_{c_0}])$. Call $A=2(K_0-\frac{{\bf 1}}{2})$ and
$B=(1-\varepsilon )B_X+\varepsilon B_{c_0}$. Now $A$ and $B$ are
convex subsets of $X$ and $B_{\varepsilon}=co(A\cup -A\cup B)$.
Observe that $\frac{A-A}{2}=K_0-K_0$ and so $\frac{A-A}{2}\subset
B_{c_0}\subset B$, from the definition of $K_0$.

Thus, in order to prove that every nonempty relatively weakly open
subset of $B_{\varepsilon}$ has
$\Vert\cdot\Vert_{\varepsilon}$-diameter 2 it is enough to prove,
from Lemma~\ref{lema}, that every nonempty relatively weakly open
subset of $\overline{co}((2K_0-{\bf 1})\cup [(1-\varepsilon )B_X
+\varepsilon B_{c_0}])$ has
$\Vert\cdot\Vert_{\varepsilon}$-diameter 2.

Pick $U$ a weakly open subset of $X$ such that $$U\cap
\overline{co}((2K_0-{\bf 1})\cup [(1-\varepsilon )B_X +\varepsilon
B_{c_0}])\neq\emptyset ,$$ then there is $g_i\in K_0$, $x_0\in
B_X$, $y_0\in B_{c_0}$ and $\lambda \in [0,1]$ such that $\lambda
(2g_i-{\bf 1})+(1-\lambda )[(1-\varepsilon )x_0+\varepsilon y_0]$
belong to $U$.

As $U$ is a norm open set, we can assume that $y_0$ has finite
support. From  Lemma~\ref{lemac0}, there is a scalar sequence
$\{t_j\}$ with $\vert t_j\vert \leq 1$ for every $j$ such that,
putting $x_j=x_0+(1-t_j)e_j$ and $y_j=x_0-(1+t_j)e_j$ for every
$j$, we have that $\{x_j\}$ and $\{y_j\}$ are weakly convergent
sequences in $B_X$ to $x_0$. We put $j_0$ such that
$e_{j}^*(y_0)=0$ for every $j\geq j_0$, then $y_0+e_j,y_0-e_j\in
B_{c_0}$ for every $j\geq j_0$. Now, again from the construction
of $K_0$, $g_i+e_{m_n+i}\in K_0$ for every $n$, and hence,
$\{g_i+e_{m_n+i}\}_n$ is weakly convergent to $g_i$.

Therefore we get for $n$ conveniently big  that
$$x:=\lambda (2(g_i+e_{m_n+i})-{\bf 1})+(1-\lambda
)[(1-\varepsilon )x_{m_n+i})+\varepsilon (y_0+e_{m_n+i}))]$$ and
$$y:=\lambda (2(g_i-{\bf 1})+(1-\lambda )[(1-\varepsilon
)y_{m_n+i})+\varepsilon (y_0-e_{m_n+i}))]$$ belong to $U$.
Therefore $$diam_{\Vert\cdot\Vert_{\varepsilon}}(U)\geq \Vert
x-y\Vert_{\varepsilon}=$$ $$\Vert 2 \lambda e_{m_n+i}+(1-\lambda
)[2 (1-\varepsilon )e_{m_n+i}+2\varepsilon
e_{m_n+i}]\Vert_\varepsilon =$$ $$ 2\Vert
e_{m_n+i}\Vert_\varepsilon \geq 2\Vert e_{m_n+i}\Vert =2\Vert
e_{m_n+i}\Vert_\infty =2 .$$ We conclude that
$diam_{\Vert\cdot\Vert_{\varepsilon}}(U)=2$.\end{proof}

The following consequence shows that there are many spaces
satisfying D2P and failing strong-D2P.

\begin{corollary} Every Banach space containing isomorphic copies
of $c_0$ can be equivalently renormed satisfying D2P and failing
strong-D2P.\end{corollary}

Finally, we get a stability property for Banach spaces with D2P
and failing strong-D2P.

\begin{corollary} The Banach spaces with D2P and failing strong-D2P are stable for $l_1$-sums.\end{corollary}

The proof of the above corollary follows from   the following
general proposition, which gives the stability under $\ell_1$-sums
of the D2P and small convex combinations of slices. In fact this
stability property holds for $1\leq p<\infty$.

\begin{proposition} Let $\{X_n\}$ be a sequence of Banach spaces
satisfying the D2P and put $Z:=\ell_1-\bigoplus_n X_n$. Assume
that $\{\varepsilon_n\}$ is a null scalars sequence such that for
every $n\in \natu$ there is a convex combination of slices in
$X_n$ with diameter, at most, $\varepsilon_n$. Then $Z$ satisfies
the D2P and $$\inf\{diam(T): T\ {\rm convex \ combination \ of \
slices\  in}\ B_Z\}=0.$$\end{proposition}

\begin{proof}

In order to prove that $$\inf\{diam(T): T\ {\rm convex\
combination\ of\ slices\ in}\ B_Z\}=0,$$ fix $n\in \natu$ and let
us see that for every slice of $B_{X_n}$ we can define a slice of
$B_Z$ with similar diameter. Consider $Z=X_n\oplus_1 Y_n$, being
$Y_n=\ell_1-\bigoplus_{k\neq n} X_k$. Let $S_n= S(B_{X_n}, x_n^*,
\alpha )$ be a slice of $B_{X_n}$ and fix $ 0 < \mu < \alpha $. We
can assume that $x_n^* \in S_{X_n^*}$. If $(x_n,y_n) \in S( B_Z,
(x_n^*,0), \mu )$, then $x_n^* (x_n)
> 1 - \mu > 1 - \alpha $ and so $\Vert x_n \Vert > 1 -\mu$. Thus
$\Vert y_n \Vert < \mu $. As a consequence, $\Vert (x_n,y_n) -
(x_n,0) \Vert < \mu $. Then we have that
\begin{equation}
 S( B_Z, (x_n^*,0), \mu ) \subset S( B_{X_n}, x_n^*, \alpha )
\times \mu B_{Y_n} .
\end{equation}
Now, if $T_n$ is a convex combination   of slices of $B_{X_n}$,
for $\mu
> 0$ small enough we get that

$$ \inf \{ diam (T): T \ \text{is a convex combination of slices
of} \ B_Z \} \leq $$ $$ diam(T_n) + 2 \mu\leq \varepsilon_n+2\mu
.$$ We conclude that $$\inf\{diam(T): T\ {\rm convex\ combination\
of\ slices\ in}\ B_Z\}=0,$$ since $\lim_n\varepsilon_n =0$.

We pass now to prove that $Z$ has D2P. As every nonempty
relatively weakly open subset of $B_{Z}$ contains a nonempty
intersection of slices in $B_Z$ \cite[Lemme 5.3]{BB}, take
$f_1,\ldots ,f_N\in S_Z$, $0<\alpha_1 ,\ldots ,\alpha_N <1$ and
consider an nonempty intersections of slices in $B_Z$
$$S=\{z\in B_Z:f_i(z)>1-\alpha_i,\ 1\leq i\leq N\}.$$ Pick $z_0\in S_Z\cap S$, then
choose $0<\varepsilon<\alpha_i$ for every $i$ so that
$f_i(z_0)>1-\alpha_i+\varepsilon$ for every $i$.

We denotes by $P_n$ the projection of $Z$ onto $\ell_1
-\bigoplus_{i=1}^n X_i$, which is a norm one projection for every
$n\in \natu$. As $f_i(z_0)>1-\alpha_i+\varepsilon$, there is
$k\in\natu$ such that
$P_k^*(f_i)(P_k(z_0))>1-\alpha_i+\varepsilon$, where $P_k^*$
denotes the transposed  projection of $P_k.$

Consider the intersections of slices in the unit ball of $Y=\ell_1
-\bigoplus_{i=1}^k X_i$ given by $T=\{y\in
B_Y:P_k^*(f_i)(y)>1-\alpha_i+\varepsilon ,\  1\leq i\leq N \}.$
Observe that $T\neq \emptyset$, since $P_K(z_0)\in T$. In order to
prove that $diam(S)=2$, fix $\rho>0$ and take $y_1,y_2\in B_Y\cap
T$ such that $\Vert y_1 -y_2\Vert>2-\rho$. This is possible,
because it is known that the finite $\ell_1$-sum of Banach spaces
with D2P has too D2P \cite{AcBeLo}. Now we see $y_1, y_2$ as
elements in Z, via the natural isometric embedding of $Y$ into
$Z$, and we have that $y_1 ,y_2 \in S$ with $\Vert y_1 -y_2\Vert_Z
>2-\rho$, hence $diam(S)\geq 2-\rho$. As $\rho$ was arbitrary, we
conclude that $diam(S)=2$.\end{proof}

Finally, we would like to pose the following
questions:\begin{enumerate}\item We don't know if $L_1$ can be
equivalently renormed  satisfying D2P so that every convex
combination of slices of its unit ball has diameter arbitrarily
samll. \item What Banach spaces can be equivalently renormed to
satisfy slice-D2P, D2P or strong-D2P? \item Is there some strongly
regular Banach space with D2P?\end{enumerate}

About the third question, recall that a Banach space $X$ is said
to be strongly regular (SR) if every closed, convex and bounded
subset of $X$ has convex combination of slices with diameter
arbitrarily small (we refer to \cite{GGMS} for background about
this topic). It is well known that every Banach space containing
isomorphic copies of $c_0$ fails to be SR. As SR is an isomorphic
property, that is independent on the equivalent norm considered in
the space, every renorming of $c_0$ fails to be SR. Also it is
known that there are SR Banach spaces so that every relatively
weakly open subset of its unit ball has diameter, at least, some
$\delta>0$, but with $\delta<2$.

About the second question, it seems natural to think that every
Banach space failing to be strongly regular can be equivalently
renormed with the strong-D2P, but we don't know if this is true.
In \cite{BeLoZo1} it is proved that every Banach space $X$, whose
dual $X^*$ fails to be strongly regular can be equivalently
renormed so that every convex combination of $w^*$-slices in the
unit ball of $X^*$ has diameter 2. Moreover, if $X$ is separable,
also it is showed there that for every $\varepsilon>0$, $X$ can be
equivalently renormed so that every convex combination of slices
in the unit ball of $X^*$ has diameter, at least, $2-\varepsilon$.

\end{document}